%% file: main.tex
\documentclass[nonatbib]{ipart}

\makeatletter
\let\c@author\relax
\makeatother

\RequirePackage{hyperref}
\usepackage{cleveref}
\usepackage{amsfonts} 
\usepackage{fullpage} 
\usepackage{graphicx}
\usepackage{enumitem} 
\usepackage[title]{appendix}
\usepackage{amssymb, commath, amsmath, hyperref, color, bm, amsthm, mathrsfs,mathtools}
\usepackage{bbm}

\usepackage{subcaption}
\usepackage[backend=biber, style=numeric]{biblatex}
\usepackage{tikz} 
\usetikzlibrary{shapes.arrows}
\usetikzlibrary{arrows,calc,spy}
\usetikzlibrary{decorations.pathreplacing,decorations.markings}
\usetikzlibrary{patterns}

\tikzset{
  on each segment/.style={
    decorate,
    decoration={
      show path construction,
      moveto code={},
      lineto code={
        \path [#1]
        (\tikzinputsegmentfirst) -- (\tikzinputsegmentlast);
      },
      curveto code={
        \path [#1] (\tikzinputsegmentfirst)
        .. controls
        (\tikzinputsegmentsupporta) and (\tikzinputsegmentsupportb)
        ..
        (\tikzinputsegmentlast);
      },
      closepath code={
        \path [#1]
        (\tikzinputsegmentfirst) -- (\tikzinputsegmentlast);
      },
    },
  },
  mid arrow/.style={postaction={decorate,decoration={
        markings,
        mark=at position .5 with {\arrow[#1]{stealth}}
      }}},
}

\tikzset{
    ncbar angle/.initial=90,
    ncbar/.style={
        to path=(\tikztostart)
        -- ($(\tikztostart)!#1!\pgfkeysvalueof{/tikz/ncbar angle}:(\tikztotarget)$)
        -- ($(\tikztotarget)!($(\tikztostart)!#1!\pgfkeysvalueof{/tikz/ncbar angle}:(\tikztotarget)$)!\pgfkeysvalueof{/tikz/ncbar angle}:(\tikztostart)$)
        -- (\tikztotarget)
    },
    ncbar/.default=0.5cm,
}
\tikzset{square left brace/.style={ncbar=0.1cm}}
\tikzset{square right brace/.style={ncbar=-0.1cm}}
\tikzset{round left paren/.style={ncbar=0.5cm,out=120,in=-120}}
\tikzset{round right paren/.style={ncbar=0.5cm,out=60,in=-60}}

\startlocaldefs
\newtheorem{theorem}{Theorem}[section]
\newtheorem{lemma}{Lemma}[section]

\theoremstyle{definition}

\newtheorem{problem}{Problem}
\newtheorem{example}{Example}[section]
\theoremstyle{remark}
\newtheorem{remark}{Remark}[section]

\newcommand{\diff}{\mathop{}\!\mathrm{d}}

\newcommand{\dx}{\diff{x}}
\newcommand{\Leb}{\mathscr{L}}
\newcommand{\id}{\mathop{}\mathopen{}\mathrm{id}}
\newcommand{\red}[1]{\textcolor{red}{#1}}
\newcommand{\defeq}{\vcentcolon=}
\newcommand{\dep}[1]{\mathrm{depth}(#1)}

\newcommand{\Pf}[1]{\mathcal{P}_{#1}}
\endlocaldefs

\begin{document}

\begin{frontmatter}

\title{Connections between Bressan's Mixing Conjecture, the Branched Optimal Transport and Combinatorial Optimization}

\begin{aug}

\author{\fnms{Bohan} \snm{Zhou}\ead[label=e1]{bhzhou@ucsb.edu}}
\address{Department of Mathematics.\\University of California, Santa Barbara, CA 93106\\\printead{e1}}
\end{aug}

\begin{abstract}
We investigate the 1D version of the notable Bressan's mixing conjecture, and introduce various formulation in the classical optimal transport setting, the branched optimal transport setting and a combinatorial optimization. In the discrete case of the combinatorial problem, we prove the number of admissible solutions is on the Catalan number. Our investigation sheds light on the intricate relationship between mixing problem in the fluid dynamics and many other popular fields, leaving many interesting open questions in both theoretical and practical applications across disciplines.
\end{abstract}




\end{frontmatter}

\input{Introduction}
\input{Bressan}

\input{OT}

\input{BOT}

\input{Combinatorics}



%
%
\printbibliography
\end{document}

%% file: Introduction.tex
\section{Introduction}

Yann Brenier and his collaborators' seminal contributions \cite{brenier1989least, Benamou2000Computational} established a profound link between optimal transport theory and fluid dynamics from its inception. Nearly a decade ago, Brenier in \cite{Brenier2015Connections} interpreted the optimal transport at a discrete level, as a quadratic assignment problem, thereby establishing connections between optimal transport, hydrodynamics, and combinatorial optimization.

In this paper, our aim is to establish connections between the \textit{mixing problem}, a significant interest in fluid dynamics, and another variant of optimal transport known as \textit{branched optimal transport}. We formulate the corresponding discrete problem as yet another instance of combinatorial optimization. 

Mixing plays a crucial role in various fluid dynamics applications. \cite{Thiffeault2012review} provides a comprehensive review of different measures on mixing and various scenarios of optimal mixing. Moreover, the impact of mixing on sampling methods, as highlighted by\cite{Ishwaran2001Gibbs}, suggests its relevance to popular generative models, as discussed in monographs such as the monographs \cite{Goodfellow2016DL}.

Our investigation stems from Bressan's notable conjecture \cite{bressan2003lemma,bressan2006prize}, which remains open despite significant progress by researchers (e.g., \cite{Alberti2019Exponential}). It is interesting to point it out that the conjecture has a combinatorics motivation. In order to mix two types of fluids (see \Cref{fig:BV}), Bressan believes that the cost to mix fluids up to some scale depends on the regularity of the velocity fields, and he proved the lower bound of the transport cost in the 1D problem. 

\input{plot/BV.tikz}

The Bressan's mixing in 1D is our starting point. In this paper, we introduce \Cref{pb: Monge} in the optimal transport formulation, \Cref{prob: BOT_formulation} in the branched optimal transport formulation and \Cref{prop: bressan1d_rooted_tree} in the combinatroics formulation. Due to the nonstandard setting of each formulation---such as the total-variation type cost function and multimarginal setting in \Cref{pb: Monge} and the addition type cost function in \Cref{prob: BOT_formulation}---directly addressing Bressan's conjecture proves challenging. Nonetheless, each new formulation presents intriguing and novel problems within its respective field. we hope this paper may draw attention from experts across various disciplines. In our new formulation, particularly in \Cref{prob: BOT_formulation} and \Cref{problem: min_rooted_tree}, the inherent graph structures may link the optimal sorting plan with the recent work on the optimal Ricci curvature on the optimal transport on graph \cite{Li2023optimal} in some sense, inviting future study.

This paper is organized as follows. In \Cref{sec: Bressan1D} we present the mathematical framework of Bressan's mixing problem in 1D. \Cref{sec: OT} then introduces its Monge's formulation. In \Cref{sec: BOT}, through a detailed example, we demonstrate that the branched optimal transport formulation offers a more natural representation. By removing dummy edges in the branched optimal transport, we finally arrive at the combinatorial formulation \Cref{sec: Comb}. At the discrete level, we prove the number of admissible solutions is the Catalan number in the main theorem \Cref{thm: P_N}. 

%% file: plot/BV.tikz
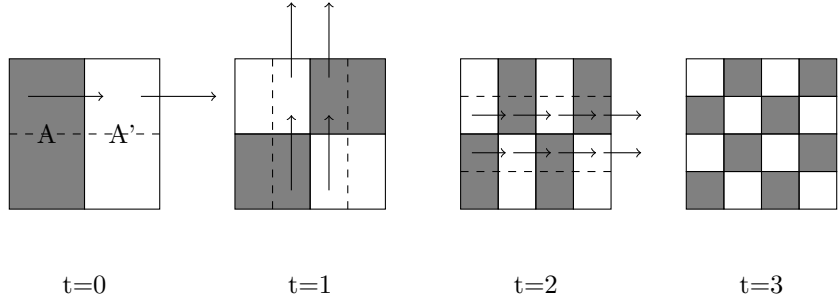
\begin{figure}[h]
\centering
\begin{tikzpicture}

\draw [fill=gray] (0,0) rectangle (1,2);
\draw [black] (1,0) rectangle (2,2);
\node at (0.5,1) {A};
\node at (1.5,1) {A'};
\draw [dashed] (0,1) -- (2,1);
\draw [->] (0.25,1.5) -- (1.25,1.5);
\draw [->] (1.75,1.5) -- (2.75,1.5);
\node at (1,-1) {t=0};

\draw [fill=gray] (3,0) rectangle (4,1);
\draw [black] (4,0) rectangle (5,1);
\draw [black] (3,1) rectangle (4,2);
\draw [fill=gray] (4,1) rectangle (5,2);
\draw [dashed] (3.5,0) -- (3.5,2);
\draw [dashed] (4.5,0) -- (4.5,2);
\draw [->] (3.75,0.25) -- (3.75,1.25);
\draw [->] (3.75,1.75) -- (3.75,2.75);
\draw [->] (4.25,0.25) -- (4.25,1.25);
\draw [->] (4.25,1.75) -- (4.25,2.75);
\node at (4,-1) {t=1};

\draw [fill=gray] (6,0) rectangle (6.5,1);
\draw [fill=gray] (6.5,1) rectangle (7,2);
\draw [fill=gray] (7,0) rectangle (7.5,1);
\draw [fill=gray] (7.5,1) rectangle (8,2);
\draw [black] (6,1) rectangle (6.5,2);
\draw [black] (6.5,0) rectangle (7,1);
\draw [black] (7,1) rectangle (7.5,2);
\draw [black] (7.5,0) rectangle (8,1);
\draw [dashed] (6,0.5) -- (8,0.5);
\draw [dashed] (6,1.5) -- (8,1.5);
\draw [->] (6.15,0.75) -- (6.6,0.75);
\draw [->] (6.15,1.25) -- (6.6,1.25);
\draw [->] (6.7,0.75) -- (7.2,0.75);
\draw [->] (6.7,1.25) -- (7.2,1.25);
\draw [->] (7.3,0.75) -- (7.8,0.75);
\draw [->] (7.3,1.25) -- (7.8,1.25);
\draw [->] (7.9,0.75) -- (8.4,0.75);
\draw [->] (7.9,1.25) -- (8.4,1.25);
\node at (7,-1) {t=2};

\draw [fill=gray] (9,0) rectangle (9.5,0.5);
\draw [fill=gray] (9.5,0.5) rectangle (10,1);
\draw [fill=gray] (10,1) rectangle (10.5,1.5);
\draw [fill=gray] (10.5,1.5) rectangle (11,2);
\draw [black] (9.5,0) rectangle (10,0.5);
\draw [black] (10,0.5) rectangle (10.5,1);
\draw [black] (10.5,1) rectangle (11,1.5);
\draw [fill=gray] (10,0) rectangle (10.5,0.5);
\draw [fill=gray] (10.5,0.5) rectangle (11,1);
\draw [black] (10.5,0) rectangle (11,0.5);
\draw [black] (9,0.5) rectangle (9.5,1);
\draw [black] (9.5,1) rectangle (10,1.5);
\draw [black] (10,1.5) rectangle (10.5,2);
\draw [fill=gray] (9,1) rectangle (9.5,1.5);
\draw [fill=gray] (9.5,1.5) rectangle (10,2);
\draw [black] (9,1.5) rectangle (9.5,2);
\node at (10,-1) {t=3};
\end{tikzpicture}
\caption{Mixing process in BV space}
\label{fig:BV}
\end{figure}

%% file: Bressan.tex
\section{Bressan's Mixing in 1D}\label{sec: Bressan1D}
To motivate readers the connection between the conjecture in the fluid dynamics with other approaches discussed in this paper, let's delve into the Bressan's mixing in 1D, aka the book-shifting problem.

Imagine a stack of books comprising both white and black books that need sorting, through a finite sequence of elementary operations, where the cost per elementary operation is $a+b$ when sorting a stack of black books of length $a$ and a stack of white books of length $b$. Bressan in \cite{bressan2003lemma} proved that the cost of sorting the books is at least on the order of $\abs{\log\varepsilon}$. Here $\varepsilon$ is the geometric mixing scale if the initial distribution, which will be defined later. However, it is still open that if there exists a sorting strategy that achieves the minimum cost (unless in discrete cases), given an initial function $\rho_0$, or even for arbitrary initial functions. 

\begin{figure}[ht]
     \centering
     \begin{subfigure}[b]{0.42\textwidth}
         \centering
         \includegraphics[width=\textwidth]{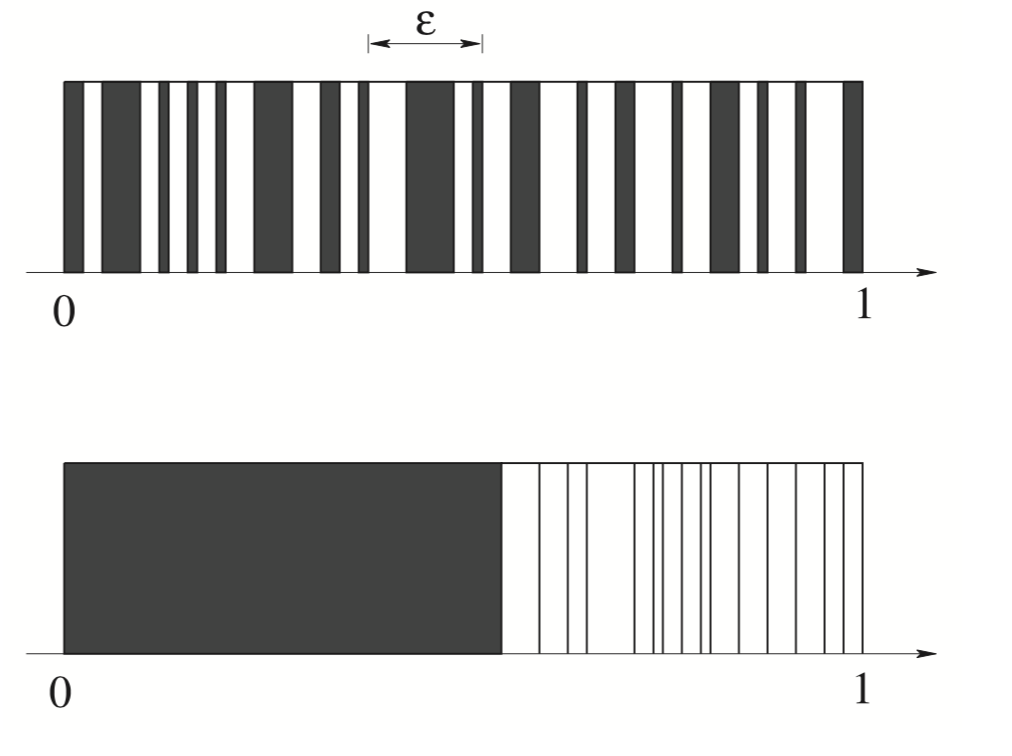}
         \caption{Book shifting}
     \end{subfigure}
     \hfill
     \begin{subfigure}[b]{0.55\textwidth}
         \centering
         \input{plot/1Dproof.tikz}
         \caption{Transform from $\rho_{k-1}$ to $\rho_{k}$\label{fig:1Dproof}}
     \end{subfigure}
     \hfill
     \caption{Bressan's mixing in 1D\label{fig:BressanLemma}}
\end{figure}

Precisely, let's consider an initial function $\rho_0:[0,1]\mapsto\{0,1\}$, where $\rho_0(x)=0$ (or $\rho_0(x)=1$) refers to the black (or white) book at position $x$. Let $m=\int_0^1 \rho_0(x)\dx$, $\mathcal{F}_m=\{\rho|\int_0^1 \rho(x)\dx=m\}$, and the terminal function $\rho_{T}$ is defined as 
    \begin{equation*}
        \rho_{T}(x)=\left\{
        \begin{aligned}
        0\qquad &\mathrm{if~} 0<x<1-m\\
        1\qquad &\mathrm{if~} 1-m<x<1.
        \end{aligned}
        \right.
    \end{equation*}

We say that $\rho_{k+1}$ is obtained from $\rho_k$ by an \textit{elementary operation} if: for some $y\in[0,1)$ and $a,b>0$, one has
\begin{equation}
\label{eq: rho_change}
    \rho_k(x)=\left\{
    \begin{aligned}
    1\qquad & \mathrm{if~} y<x<y+a,\\
    0\qquad & \mathrm{if~} y+a<x<y+a+b.
    \end{aligned}
    \right.
    \qquad
    \rho_{k+1}(x)=\left\{
    \begin{aligned}
    &0\qquad  &\mathrm{if~}& y<x<y+b,\\
    &1\qquad  &\mathrm{if~}& y+b<x<y+a+b,\\
    &\rho_k(x)\qquad  &&\mathrm{otherwise}.
    \end{aligned}
    \right.
\end{equation}
Then the cost $C(\rho_{k},\rho_{k+1})=a+b$. Note that we do not allow the reverse transposition, i.e., white books can only be transferred to the right, and black books can only be transferred to the left. 

We say that a sequence of density functions $\{\rho_0,\rho_1,\cdots,\rho_k,\cdots,\rho_n\}$ is a \textit{sorting plan} if  $\rho_k\in\mathcal{F}_m$ for any $k$, and $\rho_{k+1}$ can be obtained through an elementary operation from $\rho_{k}$, and $\rho_n=\rho_T$.

We call $\rho_0$ is $\kappa$-mixed in the \textit{geometric mixing scale} $\varepsilon$, provided with
    \begin{equation}\label{eq:geomixscale2}
    \kappa\varepsilon \leqslant \int_{y}^{y+\varepsilon} \rho_0(x)\diff{x} \leqslant (1-\kappa)\varepsilon,\qquad\forall y\in[0,1-\varepsilon]
    \end{equation}
for some constant $0<\kappa<1$ and for some scale $\varepsilon$. One may prescribe $\kappa$ first, for example $\kappa=\frac{1}{3}$, resulting in that every $\varepsilon$-interval consists of at least $\frac{1}{3}\varepsilon$ length of white (and black) books.

\begin{lemma}[\cite{bressan2003lemma}]
\label{lem: 1Dlemma}
For any sorting plan $\{\rho_0,\rho_1,\cdots,\rho_k,\cdots,\rho_n\}$, the cost of sorting satisfies:
\begin{equation}\label{eq:bressan1d}
\sum_{i=0}^{n-1}C(\rho_i,\rho_{i+1})\geqslant C(\kappa)\abs{\log\varepsilon}.
\end{equation}
\end{lemma}

For the sake of completeness, we attach the modified proof to \cref{lem: 1Dlemma}.

\begin{proof}
First, we find the estimate on $m$ in terms of $\kappa$ and $\varepsilon$.\par

On one hand, $m=\displaystyle\int_0^1 \rho_0(x)\diff{x}\leqslant \left(\int_0^1+\int_{1-\varepsilon}^{\varepsilon\lfloor\frac{1}{\varepsilon}\rfloor}\right)\rho_0(x)\diff{x}\leqslant (1-\kappa)\varepsilon \left(\left \lfloor \frac{1}{\varepsilon}\right \rfloor+1\right)$;\par

On the other hand, $m=\displaystyle\int_0^1 \rho_0(x)\diff{x}\geqslant \int_0^{\varepsilon \lfloor \frac{1}{\varepsilon}\rfloor} \rho_0(x)\diff{x}  \geqslant \kappa\varepsilon \left\lfloor \frac{1}{\varepsilon}\right\rfloor$.\par

For $s\in [\varepsilon, 1-\kappa\varepsilon\left\lfloor \frac{1}{\varepsilon}\right\rfloor]$, let $V(s)$ denote the minimum cost to get some interval of black books ($\rho(x)=0$), the length of which is greater or equal to $s$ transformed from $\rho_0$. Apparently, $V(s)$ is an increasing function.\par

First, we note that for any sorting plan $\left\{\rho_0,\rho_1,\cdots,\rho_k,\cdots,\rho_n\right\}$, we have

\begin{equation}\label{eq: 1dproof}
\sum_{i=0}^{n-1} C(\rho_i,\rho_{i+1})\geqslant V(1-m) \geqslant V(\kappa/2).
\end{equation}

Second, for $s\in [\varepsilon, 1-\kappa\varepsilon\left\lfloor \frac{1}{\varepsilon}\right\rfloor]$, assume that there is a sorting plan $\left\{\rho_0,\rho_1,\cdots,\rho_k\right\}$ that achieves the optimality of $V(s)$, and $\rho_k$ is the first density function, which has an interval of black books $\geqslant s$. Then, for $\rho_{k-1}$ there must be intervals of black books of length $\sigma$ and $s-\sigma$ for some $0<\sigma<s$, to formulate the interval of black books of length $s$. By definition,\par

\begin{equation}\label{eq: vs1}
    V(s)=\sum_{i=0}^{k-1} C(\rho_i,\rho_{i+1})\geqslant \sum_{i=0}^{k-2} C(\rho_i,\rho_{i+1})\geqslant V(\sigma) + V(s-\sigma).
\end{equation}

Note that $V(\varepsilon)\geqslant 0$, and by mathematical induction, we have

\begin{equation}\label{eq: vs2}
    V(s)\geqslant s-\varepsilon.
\end{equation}

Third, we do a more careful estimate for \eqref{eq: vs1}. We assume $s-\sigma\geqslant\varepsilon$ and the interval of length $s-\sigma$ is on the left of the interval of length $\sigma$.

This time we include the estimate for $C(\rho_{k-1},\rho_k)$. The key is to estimate the length $t$ of white books. Note that the white books, between the black books of length $s-\sigma$ and $\sigma$, cannot come from the right of the black books of length $\sigma$. And it has to include all white books inside the interval of length $s-\sigma$, where $\rho_0$ has.\par

\begin{equation}\label{eq: t}
t\geqslant \kappa\left((s-\sigma)-(1-\kappa)\varepsilon\right)\geqslant \kappa^2(s-\sigma).
\end{equation}

$$t\geqslant \kappa\left(\varepsilon \left\lfloor \frac{s-\sigma}{\varepsilon}\right\rfloor\right)$$

Hence $C(\rho_{k},\rho_{k+1})=t+\sigma \geqslant \kappa^2(s-\sigma)+\sigma\geqslant \kappa^2 s$.\par

Combine \eqref{eq: vs1}, \eqref{eq: vs2} and \eqref{eq: t}, we have
$$
\begin{aligned}
V(S)&\geqslant V(\sigma)+V(s-\sigma)+\kappa^2 s\\
&\geqslant (\sigma-\varepsilon) + (s-\sigma-\varepsilon) + \kappa^2 s\\
&\geqslant (1+\kappa^2) s -2\varepsilon,
\end{aligned}$$
By mathematical induction,
\begin{equation}\label{eq: vsn}
\begin{aligned}
    V(S)&\geqslant V(\sigma) + V(s-\sigma) + \kappa^2 s\\
    &\geqslant (1+\kappa^2) \sigma -2\varepsilon + (1+\kappa^2 )(s-\sigma)-2\varepsilon + \kappa^2 s\\
    &\geqslant (1+2\kappa^2)s - 2\cdot 2\varepsilon\\
    &\geqslant (1+n\kappa^2)s- 2^n \varepsilon.
\end{aligned}
\end{equation}
Let $g(n)=(1+n\kappa^2)s-2^n \varepsilon$, we want to find the maximum of $g(n)$ to be the lower bound of $V(S)$. By a basic calculus, its maximum achieves at $n= C(\kappa)\abs{\log\varepsilon}.$\par

Combine \eqref{eq: 1dproof} and \eqref{eq: vsn}, we have
$$\sum_{i=0}^{i=n-1}C(\rho_i,\rho_{i+1})\geqslant C(\kappa)\abs{\log\varepsilon}.$$

\end{proof}

%% file: plot/1Dproof.tikz
\begin{tikzpicture}

\draw [->] (-0.5,4) -- (10.5,4);
\node at (0,3.7) {$0$};
\node at (10,3.7) {$1$};
\draw [-] (0,4) -- (0,6);
\draw [-] (10,4) -- (10,6);
\draw [-] (0,6) -- (10,6);
\draw [fill=gray] (2,4) rectangle (4.5,6);
\draw [fill=gray] (5.5,4) rectangle (6.5,6);
\node at (3.2,5) {$s-\sigma$};
\draw [->] (2.5,5) -- (2,5);
\draw [->] (4,5) -- (4.5,5);
\node at (6,5) {$\sigma$};
\node at (5,5) {$t$};
\draw [->] (4.8,5) -- (4.5,5);
\draw [->] (5.2,5) -- (5.5,5);
\draw [->] (5.8,5) -- (5.5,5);
\draw [->] (6.2,5) -- (6.5,5);

\pattern [pattern=checkerboard,pattern color=black!30] 
(0,4) rectangle (2,6);

\pattern [pattern=checkerboard,pattern color=black!30] 
(6.5,4) rectangle (10,6);

\node at (5,3.5) {$\rho_{k-1}$};

\draw [->] (-0.5,0) -- (10.5,0);
\node at (0,-0.3) {$0$};
\node at (10,-0.3) {$1$};
\draw [-] (0,0) -- (0,2);
\draw [-] (10,0) -- (10,2);
\draw [-] (0,2) -- (10,2);
\draw [fill=gray] (2,0) rectangle (5.5,2);
\draw [fill=white] (5.5,0) rectangle (6.5,2);
\node at (3.7,1) {$s$};

\node at (6,1) {$t$};
\draw [->] (3.5,1) -- (2,1);
\draw [->] (4,1) -- (5.5,1);
\draw [->] (5.8,1) -- (5.5,1);
\draw [->] (6.2,1) -- (6.5,1);

\pattern [pattern=checkerboard,pattern color=black!30] 
(0,0) rectangle (2,2);

\pattern [pattern=checkerboard,pattern color=black!30] 
(6.5,0) rectangle (10,2);

\node at (5,-0.5) {$\rho_{k}$};

\end{tikzpicture}

%% file: OT.tex
\section{Optimal Transport Formulation}\label{sec: OT}

Following the ideas in \cite{Bianchini2006Bressan}, we are able to obtain the Monge's formulation under the special cost $c(x,y)=1_{x\neq T(x)}$ in the optimal transport, if we align a natural flow map for the elementary operation defined in \eqref{eq: rho_change}: 

\begin{equation}
\label{eq: ele_map}
T_{k+1}(x)=\left\{
\begin{aligned}
&x+b\qquad \mathrm{if~}& y<x<y+a \mathrm{~and~}& \rho_k(x)=1\\
&x-a\qquad \mathrm{if~}& y+a<x<y+a+b \mathrm{~and~}& \rho_k(x)=0\\
&x\qquad &\mathrm{otherwise}&
\end{aligned}
\right.
\end{equation}

Thus, the cost of elementary transposition can be re-calculated by:\par
\begin{equation}
    \label{eq: Bressan_cost}
    C(\rho_k,\rho_{k+1})=\int_0^1 1_{x\neq T_{k+1}(x)}\diff{x}=a+b.
\end{equation}

Let $\mathcal{T}_e$ denote the family of elementary transpositions in the form of \eqref{eq: ele_map}.

\begin{problem}[Monge's optimal transport]
\label{pb: Monge}
Under the assumption of \cref{lem: 1Dlemma}, given two Radon measures $\mu=\rho_0\Leb$ and $\nu=\rho_T\Leb$ in $\mathcal{M}_{m}([0,1])$. Let $T_0=\id$.

For some integer $k$, let $T^{(k)}=T_k\circ T_{k-1}\circ\cdots\circ T_1\circ T_0$ induce a sorting plan, where each $T_i \in \mathcal{T}_e$.

The Monge's formulation is given by
\begin{equation}
    \label{eq: Monge}
    \inf_k\inf_{T^{(k)}}\left\{M(T^{(k)})=\sum_{i=0}^{k-1} \int_0^1 1_{T^{(i)}(x)\neq T^{(i+1)}(x)}\diff{x}.\right\}
\end{equation}
\end{problem}

Suppose $T^{(*)}$ is the minimizer to \eqref{eq: Monge}, the opposite part to \cref{lem: 1Dlemma} is to check
$$M(T^{(*)})=C(\kappa)\abs{\log\varepsilon}.$$
\begin{remark}
If one fixes $k=1$, then it corresponds to the Strassen's theorem in \cite{Villani2003OT} and the total cost is just $\frac{1}{2}\norm{\mu-\nu}_{TV}$. However, \Cref{pb: Monge} is a double-minimization problem with variable marginals, and Monge's problem in the classical optimal transport is known to be restricted, it is unclear how to leverage the classical OT theory.
\end{remark}

%% file: BOT.tex
\section{Branched Optimal Transport}\label{sec: BOT}

In this section, we present another formulation in terms of the branched optimal transport. Interested readers may refer to \cite{xia2003optimal, xia2015motivations} for the rigorous and standard mathematical formulation for the branched optimal transport. Here, let's start with a series representation of the book-shifting problem.

\subsection{Series representation of book-shifting problem}
We start with a discrete version of book-shifting problem, where we replace the length of a stack of books with the number of books.

We consider all the books are placed on the locations $X=\left\{1,2,\cdots, N\right\}$. Without loss of generality, we consider the function $\rho:X\mapsto \{-1,1\}$, where $\rho(x)=-1$ refers to the black book at the position $x$, while $\rho(x)=1$ refers to the white book at the position $x$, $N$ is the total number of books. Without loss of generality, we assume the first stack is white books and the last stack is black books.
    
For a successive stack of one-type books, add their function values together and introduce the function $\rho$, which induces the alternating series:
    $$
    \{c_i\}_{i=1}^{2k}=\left\{
    \begin{aligned}
    a_{\frac{i+1}{2}}&\quad&\mathrm{i~is~odd};\\
    -b_{\frac{i}{2}}&\quad&\mathrm{i~is~even};
    \end{aligned}\right.
    $$
    where $\{a_i\}_{i=1}^k$ and $\{b_i\}_{i=1}^k$ are two positive series, satisfying $\sum_{i=1}^k (a_i+b_i)=N$. The above can be easily generated to the case of real number. 
    
    If $a_j$ permutes with $-b_j$, that is, a stack of $a_j$ white books is permuted with a stack of $b_j$ black books, then after this permutation, the number of series decreases by 2, and the operation cost is $a_j+b_j$.
    
    \input{plot/Case1.tikz}
    
    If $a_1$ permutes with $-b_1$, then after this permutation, the element $-b_1$ will not move thus will not generate any cost, the number of series decreases by 1, and its cost is $a_1+b_1$. Similar thing happens if $a_k$ permutes with $-b_k$.\par
    
    \input{plot/Case2.tikz}
    
    After at most $(2k-1)$ permutations, we finish sorting, in the form of $(-\sum_{i=1}^k b_i, \sum_{i=1}^k a_i)$.

\begin{remark}
It is worthy to note that the reverse operation (which is not allowed by Bressan) may be optimal. For example, if we have a series representation $(5,-1,1,-5)$, the optimal operation is to permute -1 and 1 first, which is a reverse operation, and we get a new series as $(6,-6)$, hence the total cost is $2+12=14$; while if we permute 5 and -1 first, then we get a new series as $(-1,6,-5)$, after one more permutation, the total cost is $6+11=17$.
\end{remark}

\subsection{Translation into branched transport model}

If we consider $\{a_i\}_{i=1}^k$ as the mass on some nodes, $\{b_i\}_{i=1}^k$ as the distance between two nodes, and we add a node $a_T$ with zero mass, as the sink. Then \Cref{fig: branched_sample1} can be interpreted as follows:\par

\input{plot/BranchedTransportGraph.tikz}

In each permutation, if $a_j$ and $a_k$ ($j<k$) merge, for example $a_j$ and $a_{j+1}$in the above graph, we call $a_k$ is a \textit{parent} of $a_j$, while $a_j$ is a \textit{child} of $a_k$. After two notes $a_j$ and $a_k$ merged, we use the parent's name and use their sum to be the mass of the new node, i.e., node $a_k$ in the following series with mass $a_k+a_j$. The cost of this permutation is $a_j+\sum_{i=j}^{k-1} b_i$.\par

In the following, we provide a case study to translate the series representation into the branched optimal transport problem: 
\begin{example}[Translate operations on series into a family tree]
\label{eg: eg1}\hfill\par
\leavevmode
The following illustrates a sorting plan on a series representation of Bressan 1D problem.
$\{12, -5, 7, -10, \red{5, -17,} 4, -5, 14, -5.\}$\par
\noindent$\to\{12, -5, 7, -27, 9, -5, \red{14, -5.}\}$\par
\noindent$\to\{\red{12, -5,} 7, -27, 9, -10, 14.\}$\par
\noindent$\to\{-5, \red{19, -27,} 9, -10, 14.\}$\par
\noindent$\to\{-32, \red{28, -10,} 14.\}$\par
\noindent$\to\{-42, 42.\}$\par
Total cost $= 22 + 19 + 17 + 46 + 38 = 142.$\par\vskip 0.1in

Now, we translate the above series into a branched transport graph.\par

In the first layer, we use $\{a_i\}_{i=1}^5$ in the \Cref{fig: translateBTG} to represent the positive terms in the series, $a_6$ to represent the sink; we align values on each direct edges between $a_j$ and $a_{j+1}$ by $b_j$. The size of each node except $a_6$ depends on its mass. The size of each edge depends on $b_j$.

First we permute 5 and -17, resulting in node $a_3$ to merge with node $a_4$. Following this permutation, we call that $a_4$ is $a_3$'s parent, and we still denote by $a_4$ for that merged node. In the branched transport graph, node $a_4$ goes down by one layer, and links with node $a_3$. The projected distance of the diagonal edge is 17.\par

We repeat the above step for each permutation, and obtain a branched transport graph \Cref{fig: translateBTG}

Now, let's verify that the cost computed from the family tree matches with the cost computed in the series operations.\par

$$
\begin{aligned}
\mathrm{Total~cost}&=(a_1+b_1)+(a_1+a_2+b_2+b_3)+(a_3+b_3)+(a_1+a_2+a_3+a_4+b_4+b_5)+(a_5+b_5)\\
&=3a_1+2a_2+2a_3+a_4+a_5+b_1+b_2+2b_3+b_4+2b_5\\
&=36+14+10+4+14+5+10+34+5+10\\
&=142.
\end{aligned}$$

\begin{figure}[h!]
    \centering
    \includegraphics[width=0.8\linewidth]{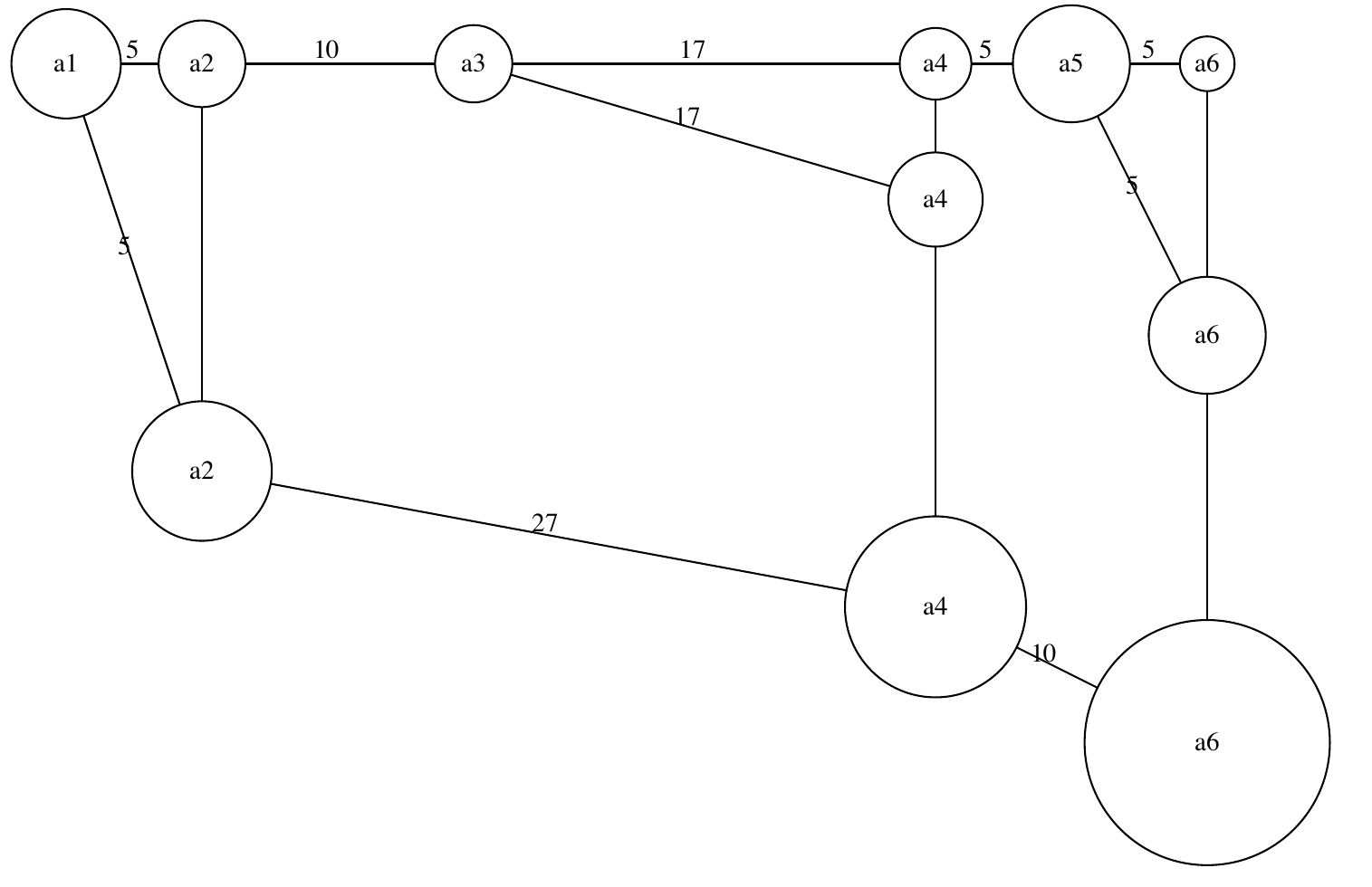}
    \caption{Branched transport graph in \Cref{eg: eg1}}
    \label{fig: translateBTG}
\end{figure}
\end{example}

\begin{problem}[Branched optimal transport]
\label{prob: BOT_formulation}
Given a space $(X,d)$, $X$ has a partially ordered ($\leqslant$) algebraic structure. The metric $d$ satisfies an additional property:
\[d(x_1,x_3)=d(x_1,x_2)+d(x_2,x_3),\quad\textrm{if}\quad x_1\leqslant x_2 \leqslant x_3.\]

Given a series of points $\{x_i\}_{i=1}^{k+1}\in X$ with a partial order, given two atomic measures $\boldsymbol{a,b}$ concentrated on $\{x_i\}_{i=1}^k\in X$ and $x_{k+1}$ correspondingly, i.e.,
\[\boldsymbol{a}=\sum_{i=1}^{k} a_i \delta_{x_i},\qquad \boldsymbol{b}=b \delta_{x_{k+1}}\]
with $b=\sum_{i=1}^{k} a_i$.\par

The measures $\boldsymbol{a,b}$ induces a piece-wise linear function $\rho(x)=\sum_{j=1}^k {\rho_j(x)}$, satisfying the mixing condition \eqref{eq:geomixscale2} for some scale $\varepsilon, \kappa$, where 
\[
\rho_j(x)=\left\{
\begin{aligned}
    &1\qquad x\in [0+\sum_1^{j-1} (a_i+d(x_i,x_{i+1})),a_j+\sum_{1}^{j-1} (a_i+d(x_i,x_{i+1}))]\\
    &0\qquad \textrm{Otherwise}.
\end{aligned}
\right.
\]

Let $E$ be the set of direct edges, $E=\{e| e^{-}=x_i, e^{+}=x_j, \mathrm{~if~} x_i\leqslant x_j\}$.\par
A \textit{transport graph} $G$ from $\boldsymbol{a}$ to $\boldsymbol{b}$, consisting of a vertex set $V(G)\subset X$, a directed edge set $E(G)$, and a weight function on edges:

$$\omega: E(G)\mapsto (0,+\infty)$$
satisfying Kirchoff's law:\par

\begin{equation}\label{eq: Kirchoff}
  \displaystyle\sum_{\substack{e\in E(G)\\ e^-=v}}w(e) = \displaystyle\sum_{\substack{e\in E(G)\\ e^+=v}}{w(e)}+\sum_{\substack{\mathrm{if~}v=x_i\\ i<k+1}}{a_i}-\sum_{\mathrm{if~}v=x_{k+1}}{b_j},\qquad \forall v\in V(G)
\end{equation}

The branched optimal transport problem is to solve:

\begin{equation}
    \label{eq: book_branch}
    \displaystyle\min_{G}\qquad C(G)=\sum_{e\in E(G)} \omega(e)+d(e).
\end{equation}

If $G^*$ is the minimizer to \eqref{eq: book_branch}, we want to show that
$$C(G^*)=C(\kappa)\abs{\log\varepsilon(\kappa)}.$$
\end{problem}

\begin{remark}
  \Cref{prob: BOT_formulation} is nonstandard in terms of branched optimal transport. The cost function in the branched optimal transport is $\omega(e)^{\alpha} d(e)$ for $\alpha\in [0,1)$. Please refer to \cite{xia2015motivations, Colombo2022Stability} for existence, regularity and stability results.   
\end{remark}

\begin{remark}
    The metric structure to this problem seems to be restricted. While one may consider concentric circles as an example. Mass $a_i$ are distributed on each circle and $x_i$ is characterized by its radius. Transporting along the circle has zero cost while transporting between concentric circles has cost depending on the difference of radiuses and masses.
\end{remark}

%% file: plot/Case1.tikz
\begin{figure}[ht!]
        \centering
        \begin{tikzpicture}
        \node at (0,0) {$a_1$};
        \node at (1,0) {$-b_1$};
        \node at (2.5,0) {$\cdots$};
        \node at (4,0) {$-b_{j-1}$};
        \node at (5,0) {$a_j$};
        \node at (6,0) {$-b_{j}$};
        \node at (7,0) {$a_{j+1}$};
        \node at (8.5,0) {$\cdots$};
        \node at (10,0) {$-b_k$};
        
        \draw [black] (3.5,-2.3) rectangle (7.4,0.3);
        \draw [-latex,thick] (5.5,-0.2) -- (5.5,-1.8);
        
        \node at (0,-2) {$a_1$};
        \node at (1,-2) {$-b_1$};
        \node at (2.5,-2) {$\cdots$};
        \node at (4.5,-2) {$-b_{j-1}-b_j$};
        \node at (6.5,-2) {$a_j+a_{j+1}$};
        \node at (8.5,-2) {$\cdots$};
        \node at (10,-2) {$-b_k$};
        \end{tikzpicture}
        \caption{$a_j$ permutes with $-b_j$}
        \label{fig: branched_sample1}
\end{figure}

%% file: plot/Case2.tikz
\begin{figure}[ht!]
        \centering
        \begin{tikzpicture}
        \node at (0,0) {$a_1$};
        \node at (1,0) {$-b_1$};
        \node at (2,0) {$a_2$};
        \node at (4,0) {$\cdots$};
        \node at (5,0) {$a_j$};
        \node at (6,0) {$-b_{j}$};
        \node at (8.5,0) {$\cdots$};
        \node at (10,0) {$-b_k$};
        
        \draw [black] (-0.3,-2.3) rectangle (2.7,0.3);
        \draw [-latex,thick] (1.5,-0.2) -- (1.5,-1.8);

        \node at (0.8,-2) {$-b_1$};
        \node at (2.1,-2) {$a_1+a_2$};
        \node at (4,-2) {$\cdots$};
        \node at (5,-2) {$a_j$};
        \node at (6,-2) {$-b_{j}$};
        \node at (8.5,-2) {$\cdots$};
        \node at (10,-2) {$-b_k$};
        \end{tikzpicture}
        \caption{Permutations on the end terms}
        \label{fig: branched_sample2}
\end{figure}

%% file: plot/BranchedTransportGraph.tikz
\begin{figure}[h]
\label{fig: branched_sample}
\centering
\begin{tikzpicture}

\draw (0,0) circle (0.4);
\node at (0,0) {$a_1$};

\draw[-] (0.4,0) -- (2,0);
\node at (1.2,0.3) {$b_1$};

\draw (2.2,0) circle (0.2);
\node at (2.2,0) {$a_2$};

\draw[-] (2.4,0) -- (3,0);
\node at (2.7,0.3) {$b_2$};

\draw[dotted] (3,0) -- (5.2,0);

\draw[-] (5.2,0) -- (5.6,0);
\node at (5.2,0.3) {$b_{j-1}$};

\draw (6.1,0) circle (0.5);
\node at (6.1,0) {$a_j$};

\draw[-] (6.6,0) -- (7.4,0);
\node at (7.0,0.3) {$b_{j}$};

\draw (7.7,0) circle (0.4);
\node at (7.7,0) {$a_{j+1}$};

\draw[-] (8.1,0) -- (8.7,0);
\node at (8.4,0.3) {$b_{j+1}$};

\draw[dotted] (8.7,0) -- (9.7,0);

\draw (10.2,0) circle (0.5);
\node at (10.2,0) {$a_{k}$};

\draw[-] (10.7,0)--(11.5,0);
\node at (11.1,0.3) {$b_{k}$};

\draw (11.6,0) circle (0.1);
\node at (11.6,0.3) {$a_{T}$};

\draw[-] (7.7,-0.4) -- (7.7,-1.2);
\path [draw=black,postaction={on each segment={mid arrow=red}}] (6.4,-0.4) -- (7.15,-1.45);

\draw (7.7,-2) circle (0.8);
\node at (7.7,-2) {$a_j+ a_{j+1}$};
\end{tikzpicture}
\caption{Branched transport graph}
\end{figure}

%% file: Combinatorics.tex
\section{Combinatorics Formulation}\label{sec: Comb}

Note that the vertical edges in \Cref{fig: translateBTG} are just dummy, we may obtain a family tree directly from the branched transport graph, as shown in \Cref{fig: familytree}.

\input{plot/FamilyTree.tikz}

Given a finite index set $[N]\defeq\{1,2,\cdots,N\}\subseteq \mathbb{N}$, we say $P:[N]\mapsto [N]$ is a \textit{parent function} if it satisfies:
\begin{enumerate}[label=\roman*)]
   \item For any $1\leqslant i\leqslant N-1$, $P(i)>i$.
   \item $P(N)=N$. 
\end{enumerate}
    
Each parent function $P$ recursively induces a \textit{generation function}  $\dep{\cdot}:[N]\mapsto\mathbb{N}$ by
\begin{align}\label{eq: generation}
    \left\{
    \begin{aligned}
    &\dep{i} = \dep{P(i)}+1;\\
    &\dep{N} = 0.
    \end{aligned}
    \right.
\end{align}
We say $P(i)$ is the parent of $i$, $i$ is the child of $P(i)$, $N$ is the root of the set $[N]$, and $\dep{i}$ is the generation number of $i$.

We begin with an example to show that not all such rooted trees corresponds a sorting plan.
\begin{example}
Let $N=4$ and $\mathrm{root}(G)=a_4$. Define a parent function by $P(1)=3$, $P(2)=4$, $P(3)=4$ and $P(4)=4$. It is easy to check this is a rooted tree, however there is no corresponding branched transport graph or a sorting plan. Since there is no way to merge white books of $a_1$ and white books of $a_3$ by skipping white books of $a_2$. And there must be an intersection on the branched transport graph, no matter how far we arrange branching nodes vertically.
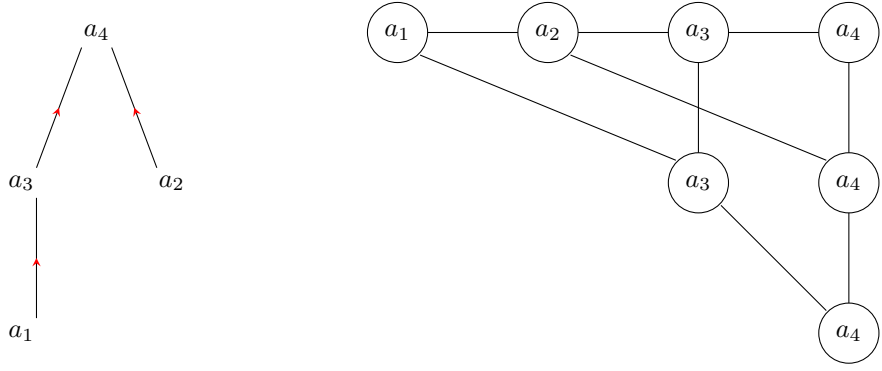
\begin{figure}[h]
\centering
\input{plot/intersection.tikz}
\caption{Left: A rooted tree. Right: Corresponding branched transport graph which is not feasible.}
\label{fig: intersect}
\end{figure}
\end{example}

\begin{lemma}\label{lem: P3}
Given a parent function $P$, for any pair $1\leqslant i<j\leqslant N$, if $P$ satisfying the following condition:
\begin{align}\label{eq: P3}
j<P(i)\Longrightarrow P(j)\leqslant P(i),
\end{align}
then the induced rooted tree $G$ induces a sorting plan. Conversely, any sorting plan induces a parent function satisfying \eqref{eq: P3}.

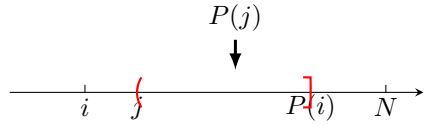
\begin{figure}[ht]
\centering
\input{plot/P3.tikz}
\caption{Once the locations of $i,j,P(i)$ are fixed, the feasible interval for $P(j)$ is given by $(j,P(i)]$.}
\label{fig: P3}
\end{figure}
\end{lemma}

\begin{proof}
Given a sorting plan, $\eqref{eq: P3}$ is satisfied naturally. Otherwise, the branched transport graph intersects, which voids the sorting plan.

Now, suppose that a parent function $P$ satisfies \eqref{eq: P3}. We note an observation that if $P(i)-i=1$, then the elementary transposition is given by switching $a_{i}$ with $b_{i}$ and merging $a_{i}$ into $a_{i+1}$. Thus, the first transposition can be always given by switching $a_{N-1}$ with $b_{N-1}$.

If $a_{N-1}$ is the only vertex whose parent is $a_{N}$, then we can consider a sub rooted tree rooted at $a_{N-1}$. Otherwise, suppose there is a sub index set $J=\left\{i\in [N]: P(i)=N\right\}$. The index set J
induces a partition of $[N]$. Let $j_1<j_2<\cdots<j_m$ denote all elements in $J$, then by \eqref{eq: P3}, for any $1\leqslant i\leqslant m-1$, $\left\{a_{j_{i}+1},a_{j_{i}+2},\cdots,a_{j_{i+1}}\right\}$ is a sub rooted tree with the root $a_{j_{i+1}}$. And $\left\{a_{1},\cdot,a_{j_1}\right\}$ is a sub rooted tree with the root $a_{j_1}$. 

By this divide-and-conquer techniques, we recursively define a sorting plan.
\end{proof}

The next question is how to compute the transport cost from the rooted tree. Recall the cost is given in \eqref{eq: book_branch}. We rewrite \eqref{eq: book_branch} by noting that an edge $e$ is given $e^-=a_{i}$ and $e^+=a_{j}$, the length of which is given by $\sum_{k=i}^{j-1}b_k$.
\begin{equation}\label{eq: re-arrange_P}
\begin{aligned}
    \mathrm{Cost}(P)&=\sum_{e\in E}\omega(e)+d(e)\\
    &=\sum_{a_{}\in V}\omega(a_{})\dep{a_{}}+\sum_{e\in E}d(e)\\
    &=\sum_{i=1}^N a_i\dep{i}+\sum_{i=1}^N\sum_{e^-=a_{i}}d(e)\\
    &=\sum_{i=1}^{N-1} a_i\dep{i}+\sum_{i=1}^{N-1}\sum_{j=i}^{P(i)-1}b_j\\
    &=\sum_{i=1}^{N-1}\left[a_i\dep{i}+b_i\sum_{j=1}^i H(P(j)-i)\right],
\end{aligned}
\end{equation}
where $H(x)$ is the Heaviside function, given by $H(x)=1$ for $x>0$ and $H(x)=0$ for $x\leqslant 0$. Recall the generation function $\dep{\cdot}$ is induced by $P$ via \eqref{eq: generation}.

\begin{problem}[A rooted tree formulation]\label{prop: bressan1d_rooted_tree}
Given a finite sequence of atomic measures $\left\{a_{i}\right\}_{i=1}^N$ supported on a metric space $(\mathbb{R}^d,d_V)$. We consider the minimization problem among all admissible rooted trees of $\{a_{i}\}$ with the root $a_{N}$, that is,
\begin{align}\label{problem: min_rooted_tree}
    \mathrm{Minimize}\qquad \mathrm{Cost}(P)=\sum_{i=1}^{N-1}\left[a_i \dep{i}+b_i\sum_{j=1}^i H(P(j)-i)\right],
\end{align}
among $P\in \Pf{N}$, where $\Pf{N}$ denote the collection of parent functions defined on $[N]$ satisfying \eqref{eq: P3}.
\end{problem}

First $\Pf{N}$ is not convex even if one relaxes to $P:\mathbb{R}\mapsto \mathbb{R}$.

Second, in order to estimate the set of $\Pf{N}$, we need the following lemma, which is in fact a rephase of \Cref{lem: P3}.

\begin{lemma}\label{lem: P3_ver2}
Given $P\in\Pf{N}$, let $G$ be the induced rooted tree. If $i$ and $j$ are consecutive children of the same parent, that is, $P(i)=P(j)$ and for all $i<k<j$, $P(k)\neq P(i)$. Then let $J=\{i+1,\cdots,j\}$ and define $Q(k)=P(k)$ for $k=i+1,\cdots,j-1$ and $Q(j)=j$, then $Q\in \Pf{j-i}$ and $J$ is a rooted tree with the root $j$ and the orientation induced by $Q$.
\end{lemma}

For any $P\in\Pf{N}$, it is clearly that it defines $\dep{\cdot}$ uniquely. Conversely, we have
\begin{lemma}\label{lem: NtoP}
For any $\dep{\cdot}:[N]\mapsto\mathbb{N}$, it induces at most a $P\in\Pf{N}$.
\end{lemma}
\begin{proof}
We may assume $\dep{N}=0$. For all $i$ such that $\dep{i}=1$, $i$ is a child of $N$, that is, $P(i)=N$. We take this index set as $I$.

Now we let $\dep{k}=\dep{k}-1$ for all $k\in [N]$. For all $j$ such that $\dep{j}=1$, by \Cref{lem: P3_ver2}, we know if $j$ is less than all indexes in $I$, then $P(j)=i$; if $j$ is between two consecutive $i_{k}$ and $i_{k+1}$ in $I$, then $P(j)=i_{k+1}$; otherwise, $\dep{\cdot}$ cannot induce an admissible parent function.

We recursively define the $P$ until $\dep{i}\leqslant 0$ for all $i$. By our construction, it is unique.
\end{proof}

Finally, we reach our main theorem in the combinatoric approach.

\begin{theorem}[Total number of the set $\Pf{N}$]\label{thm: P_N}\footnote{Personal Communication with Dr. Dun Qiu.}
$\#(\Pf{N})=a_N=C_{N-1}\defeq\frac{1}{N}\binom{2N-2}{N-1}$ given by the $(N-1)$-th Catalan number $C_{N-1}$.
\end{theorem}

In order to prove \Cref{thm: P_N}, we need a technical lemma.

\begin{lemma}\label{lem: Q_nomatter_last}
Given a constant $K>N$ and a collection $\mathcal{Q}$ of functions $Q:[N]\mapsto [N]\cup \{K\}$ satisfying
\begin{enumerate}[label=\alph*)]
    \item $Q(i)>i$ for all $i\in [N]$;
    \item $Q(i)\leqslant N$ for all $1\leqslant i\leqslant N-1$;
    \item $Q(N)=K$;
    \item For any pair $1\leqslant i < j\leqslant N$, if $j<Q(i)$, then $Q(j)\leqslant Q(i)$.
\end{enumerate}
Then $\#(\mathcal{Q})=\#(\Pf{N})$.
\end{lemma}

\begin{proof}
Given a $P\in\Pf{N}$, we define $Q$ by
\begin{align*}
    Q(i)=\left\{
    \begin{aligned}
    &P(i);\qquad& 1\leqslant i\leqslant N-1;\\
    &K;\qquad& i=N.
    \end{aligned}
    \right.
\end{align*}
By the definition of parent functions, \textit{a)} is satisfied. Since the range of $P$ is $[N]$, \textit{b)} is satisfied. \textit{c)} is trivial and \textit{d)} holds as well, because $Q(i)\leqslant N$ which forces $j<N$, for $1\leqslant i<j<N$, by \eqref{eq: P3} we obtain the result.

On the other hand, given a $Q\in\mathcal{Q}$, we define $P$ by
\begin{align*}
    P(i)=\left\{
    \begin{aligned}
    &Q(i);\qquad& 1\leqslant i\leqslant N-1;\\
    &N;\qquad& i=N.
    \end{aligned}
    \right.
\end{align*}

By \textit{b)} and $P(N)=N$, conditions in the definition of parents functions are satisfied. For $1\leqslant i<j<N$, \eqref{eq: P3} holds due to \textit{d)}. For $1\leqslant i<N$ and $j=N$, if $Q(i)>j=N$, then $i>N$, contradiction! Therefore, \eqref{eq: P3} always holds.
\end{proof}

\begin{proof}[Proof of \Cref{thm: P_N}]
It is trivial to check the formula holds for $k=1,2$. Now we assume the formula holds for $k\leqslant N-1$ and we check the case for $k=N$.

Since $P(1)$ has $N-1$ possible vlaues $\{2,3,\cdots,N\}$. Suppose $P(1)=j$, then by \Cref{lem: P3}, $\{j,j+1,\cdots,N\}$ forms a sub rooted tree. By the assumption of mathematical induction, there are $a_{N-(j-1)}$ such parent functions satisfying \eqref{eq: P3} that maps the subset $\{j,j+1,\cdots,j\}$ to $\{j,j+1,\cdots,j\}$. By \Cref{lem: Q_nomatter_last}, there are $a_{j-1}$ such parent functions satisfying \eqref{eq: P3} that maps the subset $\{2,3,\cdots,j-1\}$ to $\{2,3,\cdots,j\}$. As a result, we have
\begin{align}\label{eq: catalan}
    a_N = \sum_{j=2}^N a_{j-1} a_{N-j+1}=\sum_{j=2}^N C_{j-2}C_{N-j}=\sum_{i=0}^{N-2}C_i C_{N-i}=C_{N-1}.
\end{align}
\end{proof}

%% file: plot/FamilyTree.tikz
\begin{figure}[h]
\centering
\begin{tikzpicture}
\node at (0,0) {$a_6$};
\node at (-1,-1) {$a_4$};
\node at (1,-1) {$a_5$};
\node at (-2,-2) {$a_2$};
\node at (0,-2) {$a_3$};
\node at (-3,-3) {$a_1$};
\path [draw=black,postaction={on each segment={mid arrow=red}}]  (-0.8,-0.8)-- (-0.2,-0.2);
\path [draw=black,postaction={on each segment={mid arrow=red}}] (0.8, -0.8)--(0.2, -0.2) ;
\path [draw=black,postaction={on each segment={mid arrow=red}}]  (-1.8, -1.8)-- (-1.2,-1.2);
\path [draw=black,postaction={on each segment={mid arrow=red}}]  (-0.2,-1.8)-- (-0.8,-1.2);
\path [draw=black,postaction={on each segment={mid arrow=red}}] (-2.8,-2.8) -- (-2.2,-2.2) ;
\end{tikzpicture}
\caption{Family Tree in \Cref{eg: eg1}}
\label{fig: familytree}
\end{figure}

%% file: plot/intersection.tikz
\begin{tikzpicture}
\draw (0,0) circle (0.4);
\draw (2,0) circle (0.4);
\draw (4,0) circle (0.4);
\draw (6,0) circle (0.4);
\node at (0,0) {$a_{1}$};
\node at (2,0) {$a_{2}$};
\node at (4,0) {$a_{3}$};
\node at (6,0) {$a_{4}$};
\draw (0.4,0)--(1.6,0);
\draw (2.4,0)--(3.6,0);
\draw (4.4,0)--(5.6,0);

\draw (4,-2) circle (0.4);
\draw (4,-0.4)--(4,-1.6);
\node at (4,-2) {$a_{3}$};
\draw (0.3,-0.3)--(3.7,-1.7);

\draw (6,-2) circle (0.4);
\draw (6,-0.4)--(6,-1.6);
\node at (6,-2) {$a_{4}$};
\draw (2.3,-0.3)--(5.7,-1.7);

\draw (6,-4) circle (0.4);
\draw (6,-2.4)--(6,-3.6);
\node at (6,-4) {$a_{4}$};
\draw (4.3,-2.3)--(5.7,-3.7);

\node at (-4,0) {$a_{4}$};
\node at (-5,-2){$a_{3}$};
\node at (-3,-2) {$a_{2}$};
\node at (-5,-4) {$a_{1}$};
\path [draw=black,postaction={on each segment={mid arrow=red}}] 
(-4.8,-1.8)--(-4.2,-0.2);
\path [draw=black,postaction={on each segment={mid arrow=red}}] 
(-3.2,-1.8)--(-3.8,-0.2);
\path [draw=black,postaction={on each segment={mid arrow=red}}] 
(-4.8,-3.8)--(-4.8,-2.2);

\end{tikzpicture}

%% file: plot/P3.tikz
\begin{tikzpicture}[>=stealth]

\draw[->] (0,0)--(5.5,0);
\node at (5,-0.2) {$N$};
\draw (1,0)--(1,0.1);
\draw (4,0)--(4,0.1);
\draw (5,0)--(5,0.1);
\draw (1.7,0)--(1.7,0.1);
\node at (1,-0.2) {$i$};
\node at (4,-0.2) {$P(i)$};
\node at (1.7,-0.2) {$j$};
\draw[-latex,line width=0.4mm] (3,0.7)--(3,0.3);
\node at (3,1) {$P(j)$};
\draw [red, thick] (1.75,-0.2) to [round left paren] (1.75,0.2);
\draw [red, thick] (3.9,-0.2) to [square right brace] (3.9,0.2);
\end{tikzpicture}